\documentclass[12pt,twoside]{amsart}
\usepackage[latin1]{inputenc}
\usepackage{times}
\usepackage{amsthm}
\usepackage{amssymb, verbatim}
\usepackage{amsmath}
\usepackage{mathrsfs}

\topmargin = - 40 pt %distanza dal margine superiore 0= un bel po` negativo diminuisce
\newcommand {\1}    {{\bf 1}}

\newcommand     {\sd}   {{\mathbb n}}

\newcommand     {\CD}   {{\C[\partial]}}
\newcommand     {\Cd}   {{\C[\partial]}}

\newcommand     {\C}    {{\mathbb C}}
\newcommand     {\N}    {{\mathbb N}}
\newcommand     {\Z}    {{\mathbb Z}}

\newcommand     {\Fou}    {{\mathcal F}}
\newcommand     {\ad}   {\mbox{ad\,}}

\newcommand     {\End}  {\mbox{End}}

\newcommand     {\lie} {{\mbox{\scriptsize {\it Lie}}}}

\newcommand     {\Tor}  {{\mbox{\it Tor }}}

\newcommand     {\open} {``}
\newcommand     {\close}{''}
\renewcommand     {\tt}   {\otimes}

\DeclareMathOperator{\gc}{gc}
\DeclareMathOperator{\id}{id}

\DeclareMathOperator{\chom}{Chom}
\DeclareMathOperator{\cend}{Cend}

\DeclareMathOperator{\Nil}{Nil}

\DeclareMathOperator{\spn}{span}

\DeclareMathOperator{\Hom}{Hom}
\DeclareMathOperator{\Aut}{Aut}

\newtheorem{thm}{Theorem}[section]

\newtheorem{lemma}{Lemma}[section]
\newtheorem{prop}{Proposition}[section]
\newtheorem{cor}{Corollary}[section]

\theoremstyle{definition}

\newtheorem{ex}{Example}[section]

\newtheorem{rem}{Remark}[section]

\newcommand{\kk}    {\mathbf{k}}       % a field
\newcommand{\dd}    {\mathfrak{d}}

\newcommand{\mb}[1]{\mathbb #1}

\newcommand{\mfr}[1]{\mathfrak #1}

\newcommand{\ot}[2]{ #1 \otimes #2}
\newcommand{\oH}[1]{{\otimes}_{H} #1}
\newcommand{\pbw}{\frac{ {{\partial}_{1}}^{i_1} \dots {{\partial}_{N}}^{i_N}}{i_1 ! \dots i_N !}}
\newcommand{\buni}[1]{{\partial}^{(#1)}}

\begin{document}

\title{A root space decomposition for finite vertex algebras}
\author[A.~D'Andrea]{Alessandro D'Andrea}
\author[G.~Marchei]{Giuseppe Marchei}
\thanks{The first author was supported by AST fundings from ``La Sapienza'' University.}
\address{Dipartimento di Matematica, Universit\`a degli Studi di
Roma ``La Sapienza''\\ P.le Aldo Moro, 5 -- 00185 Rome, Italy}
\email{dandrea@mat.uniroma1.it}

\begin{abstract}
Let $L$ be a Lie pseudoalgebra, $a \in L$. We show that, if $a$ generates a (finite) solvable subalgebra $S = \langle a \rangle \subset L$, then one may find a lifting $\bar a \in S$ of $[a] \in S/S'$ such that
$\langle \bar a \rangle$ is nilpotent.

We then apply this result towards vertex algebras: we show that every finite vertex algebra $V$ admits a decomposition into a semi-direct product $V = U \sd N$, where $U$ is a subalgebra of $V$ whose underlying Lie conformal algebra $U^\lie$ is a nilpotent self-normalizing subalgebra of $V^\lie$, and $N = V^{[\infty]}$ is a canonically determined ideal contained in the nilradical $\Nil V$.
\end{abstract}
\maketitle
\tableofcontents

\section{Introduction}
Let $H$ be a Hopf algebra. One may make \cite{BD} the class of left $H$-modules into a pseudotensor category $\mathcal M^*(H)$ in a non-standard way; an algebra object in $\mathcal M^*(H)$ is then a {\em pseudoalgebra} over $H$.

In \cite{BDK1}, the notion of Lie pseudoalgebra over a cocommutative Hopf algebra was introduced and studied. Not surprisingly, the study of finite Lie pseudoalgebras amounts to investigating commutator properties of families of {\em pseudolinear endomorphisms}. When $M$ is a finite (i.e., finitely generated) left $H$-module, the space of all pseudolinear endomorphisms of $M$ can be given a natural Lie pseudoalgebra structure, denoted by $\gc M$.

One of the Lie theoretic features of the pseudo-version of linear algebra is that an analogue of Lie Theorem holds. However, a pseudolinear endomorphism may fail to self-commute; as a consequence, not all pseudolinear endomorphisms can be made to \open stabilize a flag.\close\,
This can be made precise: if $f \in \gc M$, where $M$ is a finite $H$-module, then $f$ can be put in {\em upper triangular form} if and only if $f$ generates a solvable subalgebra $\langle f \rangle$ of the Lie pseudoalgebra $\gc M$; moreover, its action decomposes $M$ into a direct sum of generalized eigenspaces if and only if $\langle f \rangle$ is nilpotent. There are examples of $f$ such that $\langle f \rangle$ is not solvable. One may also choose $f$ so that $\langle f \rangle$ is solvable but not nilpotent.\\

The first part of this paper is devoted to showing that whenever $S = \langle f \rangle \subset \gc M$ is solvable, $f$ is not too far from generating a nilpotent subalgebra of $\gc M$. More precisely, one may always find $\bar f \equiv f \mod S'$ such that $\langle \bar f \rangle$ is nilpotent. Therefore, even though $f$ may fail to decompose $M$ into a direct sum of generalized eigenspaces, a (non-unique) suitable {\em modification} of $f$ certainly does. We expect this fact to be useful towards the study of some class of subalgebras of the Lie pseudoalgebra $\gc M$, where $M$ is a finite $H$-module, e.g., subalgebras of $\gc M$, all of whose nonzero subalgebras contain a nonzero self-commuting element.

In the second half of the paper, we employ this fact towards characterizing finite vertex algebras by studying the adjoint representation of the underlying Lie conformal algebra. Vertex algebras that are of interest in physics are very large objects, and are typically graded vector spaces of superpolynomial growth. It is well known that finite-dimensional vertex algebras collapse to differential commutative algebra structures; however, infinite-dimensional examples of low growth are less well understood.
One of the authors showed in \cite{nilpo} that all finite vertex algebras $V$ possess a solvable underlying Lie conformal algebra $V^\lie$, and that all generalized weight space (of nonzero weight) with respect to the adjoint action of any subalgebra of $V^\lie$ are nil-ideals of the vertex algebra structure.

A more precise description can be obtained by mimicking the root space decomposition technique in this new setting: if $V$ is a finite vertex algebra, choose a generic element in $V^\lie$, and modify it so that it generates a nilpotent subalgebra of $V^\lie$. Then, decompose $V$ into direct sum of generalized weight spaces. All nonzero weights result in abelian vertex ideals, whereas the generalized $0$-weight space is a vertex subalgebra $U$ of $V$ with the property that $U^\lie$ is nilpotent and self-normalizing in $V^\lie$, i.e., it is a {\em Cartan subalgebra} of $V$. Then $V$ decomposes into a semidirect product of $U$ with a canonically determined abelian ideal $N$; namely, $N$ is the ideal of $V$ on which the central series of $V^\lie$ stabilizes. Finally, we show by an explicit example that $N$ may fail to vanish. This shows that there exist finite vertex algebras whose underlying Lie conformal algebra is not nilpotent.

The general philosophy is that vertex algebras naturally tend to be very large objects. Because of this, the algebraic requirement that finitely many quantum fields close, up to $\CD$-linear combination, under normally ordered product and $\lambda$-bracket, forces some form of nilpotence on the structure; we describe the exact form of this nilpotence in Theorem \ref{fVAdec}.

\section{Preliminaries on Lie pseudoalgebras}
In this paper we will work over an algebraically closed field $\kk$ of zero characteristic. Unless otherwise specified, all vector spaces, linear maps and tensor product will be considered over $\kk$.
\subsection{Hopf algebras and Lie pseudoalgebras}
Let $H$ be a cocommutative Hopf algebra \cite{Sw} with coproduct $\Delta$, $\Delta(h) =  \ot {h_{(1)}}{h_{(2)}}$, counit $\epsilon$ and antipode $S$.

The tensor product $H \tt H$ can be made into a right $H$-module by $\alpha. h = \alpha \Delta(h),$ where $\alpha \in H \tt H, h \in H$. If $L$ is a left $H$-module, it makes then sense to consider $(H \tt H)\oH L$, along with its natural left $H \tt H$-module structure.

A \emph{Lie pseudoalgebra} over $H$ is a (left) $H$-module $L$ together with a \emph{pseudobracket}, i.e., an $\ot
HH$-linear map
\begin{equation*}%\label{pseudo}
\begin{array}{cccc}
[\,\, \ast \,\,] : & L \otimes L & \longrightarrow & {(H \otimes H)} {\otimes}_{H} L\\
\quad & a \otimes b & \longmapsto & [a \ast b ]
\end{array}
\end{equation*}
satisfying \emph{skew-commutativity}
\begin{equation*}%\label{skew}
[b \ast a ] = -(\sigma \oH {{\id}_{L}}) [a \ast b],
\end{equation*}
and the
\emph{Jacobi identity}
\begin{equation}\label{jacobi}
[[ a \ast b] \ast c] = [ a \ast [ b \ast c ]] - (( \sigma \otimes
\id) {\otimes}_{H} \id ) [ b \ast [ a \ast c ]],
\end{equation}
for all choices of $a, b, c \in L$.
Here, $\sigma: H \tt H \to H \tt H$ denotes the permutation of factors, $\sigma ( \ot hk) =  \ot kh$, and \eqref{jacobi} takes place in $( H \otimes H \otimes H) \oH L$, once we extend the pseudobracket
so that
$$[ ((h \tt k) \oH r)  \ast s ] = \sum\limits_{i} ( h \tt k \tt 1) (\ot \Delta \id)
( \ot {f^i} {g^i} ) \oH t_i,$$
$$[ r \ast ((h \tt k) \oH s) ] = \sum\limits_{i} (1 \tt h \tt k) (\ot \id \Delta)
( \ot {f^i} {g^i} ) \oH t_i,$$ if $[r \ast s] = \sum\limits_{i} (
\ot {f^i} {g^i} ) \oH t_i$, where $f^i, g^i, h, k \in H$, $r, s, t_i \in L$.

If $L, M$ are Lie pseudoalgebras over $H$, then an $H$-linear map $f: L \to M$ is a {\em Lie pseudoalgebra homomorphism} if $[ f(a) \ast f(b)] = ((\id \tt \id) \oH f) [a \ast b],$ for all $a, b \in L$. A Lie pseudoalgebra $L$ is \emph{finite} if it is finitely generated as an $H$-module.
\begin{ex}
If $H=\kk$ then $\ot HH \simeq H$ and $\Delta =\id$. In this case the notion of Lie pseudoalgebra over $\kk$ is equivalent to the ordinary notion of a Lie algebra.
\end{ex}
\begin{ex}\label{conformal}
Let $\mathfrak{d} = \kk \partial$ be a one dimensional Lie algebra. Then $H=\mathcal{U} (\mathfrak{d})=\kk [\partial]$ has a standard cocommutative Hopf algebra structure. In this case the axioms of Lie pseudoalgebra over $H$ are equivalent to the axioms of Lie conformal algebra \cite{DK, K}. The equivalence between pseudobracket and $\lambda$-bracket is given by
$$[ a \ast b ] = \sum\limits_{i} P_i ( \ot \partial 1, \ot 1 \partial)
\oH c_i \,\, \iff %\longleftrightarrow
\,\, [a _{\lambda} b ] =
\sum\limits_{i} P_i (-\lambda,\partial+\lambda) c_i.$$
\end{ex}

\subsection{Hopf algebra notations}
Throughout the rest of the paper $\dd$ will denote a finite-dimensional Lie algebra, and $H=\mathcal{U} (\mathfrak{d})$ its universal enveloping algebra.

$H$ is a Noetherian domain and possesses a standard Hopf algebra structure satisfying
$$\Delta (\partial) = \ot \partial 1 + \ot 1 \partial, \quad \,\, S (
\partial )=-\partial, \qquad \quad \partial \in \mathfrak{d}.$$
If $\dim \mfr d=N$ and ${\{\partial_i\}}_{i=1}^{N}$ is a basis of $\mfr d$ then
\begin{equation*}%\label{PBW}
\buni I = \pbw, \qquad \qquad I=(i_1,\dots,i_N) \in \N^{N},
\end{equation*}
is a $\kk$-basis of $H$, by the Poincar\`{e}-Birkoff-Witt Theorem. The coproduct satisfies
\begin{equation}\label{delta}
\Delta ({\buni I}) = \sum\limits_{J+K=I} \ot {\buni J}{\buni K}.
\end{equation}
Recall that $H$ has a canonical increasing filtration given by
$$F^{n}H = \spn_{\kk} \{ \buni I \mid | I | \leq n \}, \qquad
\,\, n=0,1,2,\dots$$
where $|I|=i_1+\dots+i_N$ if $I=(i_1,\dots,i_N)$. This filtration satisfies $F^{-1} H=\{ 0 \}$, $F^{0} H = \kk$,
$F^{1} H = \kk \oplus \mfr d$. We will say that elements in $F^i H \setminus F^{i-1} H$ have degree
$i$. Due to \eqref{delta}, $\Delta (h) - 1 \tt h \in H \tt F^{i-1}H$ if $h$ has degree $i$.

\begin{rem}
It is easy to check that
$$h \tt k = (h S(k_{(1)})\tt 1) \cdot  \Delta(k_{(2)}),$$
for all $h, k \in H$, hence every element of $H \tt H$ can be expressed in the form $\sum_i (h_i \tt 1) \Delta (l_i)$, where $h_i, l_i \in H$.

Similarly, whenever $M$ is an $H$-module, elements from $(H \tt H) \oH M$ can be {\em straightened} to the form $\sum_i (h_i \tt 1) \oH m_i$. Notice that both the $h_i$ and the $m_i$ can be chosen to be linearly independent.
\end{rem}
\begin{lemma}\label{HoHLtoHA}
The linear map $\tau_M: (H \tt H)\oH M \to H \tt M$ defined by $\tau_M((h \tt k) \oH m) = hS(k_{(1)}) \tt k_{(2)}m$ is an isomorphism of vector spaces.
\end{lemma}
\begin{proof}
It is clearly well defined, and $h \tt m \mapsto (h \tt 1) \oH m$ is its inverse.
\end{proof}
\begin{rem}
When $M = H$, the above lemma shows invertibility of the map $h \tt k \mapsto h S(k_{(1)}) \tt k_{(2)}$. Its inverse $\Fou: h \tt k \mapsto h k_{(1)} \tt k_{(2)}$ is a linear endomorphism of $H \tt H$, called {\em Fourier transform} in \cite{BDK1}.
\end{rem}
\begin{cor}\label{coefficienti}
Let $M$ be an $H$-module, $\alpha = \sum_i (f^i \tt g^i) \oH m_i \in (H \tt H) \oH M$, $\gamma \in \Hom_\kk(H, H)$. Then the element
$$\alpha_\gamma = \sum_i \gamma(f^i S(g^i_{(1)})) g^i_{(2)} m_i \in M$$
is well defined.
\end{cor}
\begin{proof}
The map $\tau_M$ from the previous lemma maps $\alpha \in (H \tt H) \oH M$ to $\sum_i f^i S(g^i_{(1)})) \tt g^i_{(2)} m_i \in H \tt M$. The element $\alpha_\gamma$ is then obtained by applying $\mu \circ (\gamma \tt \id_M)$, where $\mu: H \tt M \to M$ is the $H$-module structure map.
\end{proof}

Elements of the form $\alpha_\gamma$ are called {\em coefficients} of $\alpha$.
\begin{rem}\label{coeff}
It is worth noticing that if $\alpha = \sum_i (h^i \tt 1) \oH c_i$, and the $h^i$ are linearly independent over $\kk$, then all elements $c_i$ can be realized as coefficients of $\alpha$; namely $c_i = \alpha_{\gamma_i}$, where $\gamma_i(h^j) = \delta_i^j$. In particular, $\alpha$ lies in $(H \tt H) \oH S$, where $S$ is an $H$-submodule of $L$, if and only if $c_i \in S$ for all $i$.
\end{rem}
Let $L$ be a Lie pseudoalgebra over $H$. For any choice of $A, B \subset L$, set $[A, B]$ to be the smallest
$H$-submodule $S \subset L$ such that $[a \ast b] \in (H \tt H) \oH S$ for all $a \in A, b \in B$. Due to Remark \ref{coeff}, $[A, B]$ is the $H$-submodule generated by coefficients of all $[a \ast b], a \in A, b \in B$. A subspace $S \subset L$ is a {\em subalgebra} of $L$ if $[S, S] \subset S$. If $X$ is a subset of the Lie pseudoalgebra $L$, then $\langle X \rangle$ denotes the {\em subalgebra generated by $X$}, i.e., the smallest subalgebra of $L$ containing $X$.

Define the \emph{derived series} of $L$ as $L^{(0)} = L$, $L^{(1)} = [L, L]$, $L^{(n+1)} = [L^{(n)}, L^{(n)}]$. Similarly, the \emph{central series} of $L$ is defined by $L^{[0]} =L$, $L^{[1]} = [L, L]$, $L^{[n+1]} = [L, L^{[n]}]$. A Lie pseudoalgebra $L$ is {\em solvable} (resp. {\em nilpotent}) if $L^{(n)}$ (resp. $L^{[n]}$) equals $(0)$ for some $n$; $L$ is \emph{abelian} if the {\em derived subalgebra} $L' = [L, L]$ equals $(0)$.

An {\em ideal} of a Lie pseudoalgebra $L$ is a subspace $I$ such that $[L, I] \subseteq I$. The {\em centre} $Z(L)$ of $L$ is the space of all elements $z \in L$ such that $[z, L] = (0)$. Every $H$-submodule of $Z(L)$ is an ideal. If $N$ is a central ideal of $L$, then $L$ is nilpotent if and only if $L/N$ is nilpotent.

\begin{lemma}\label{alfaoHemme}
Let $H$ be a cocommutative Hopf algebra, $M$ an $H$-module. Assume $\alpha \in H \tt H$ and $m \in M$ is not a torsion element. Then $\alpha \oH m=0$ if and only if $\alpha=0$.
\end{lemma}
\begin{proof}
Let $Hm$ be the cyclic module generated by $m$. Assuming that $m$ is not torsion is equivalent to requiring that the map $\phi: H \to Hm$, $\phi(h) = hm$ is an isomorphism of vector spaces. Let
$$\alpha=\sum_i h^i \otimes k^i \in \ot HH.$$
By Lemma \ref{HoHLtoHA} we have
\begin{equation*}%\label{alfaprimo}
\tau_M ( (\alpha \oH m ))=\sum_i h^i S({k^i}_{(1)}) \otimes {k^i}_{(2)} m \in H \otimes Hm\subset H \tt M.
\end{equation*}
Applying the invertible map $\Fou \circ (\id_H \tt \phi^{-1}): H \tt Hm \to H \tt H$ to this element gives back $\alpha$.
\end{proof}

\subsection{The general linear pseudoalgebra}
Let $L$ be a Lie pseudoalgebra over $H$. A \emph{representation} of $L$, or $L$-\emph{module}, is an $H$-module $V$ endowed with an $H \tt H$-linear {\em action}
$$ \ot L V \ni \ot a v \mapsto a \ast v \in (\ot H H ) \oH V,$$
such that, for every $a,b \in L$, $v \in V$,
\begin{equation*}%\label{action}
[ a \ast b ] \ast v = a \ast ( b \ast v ) - (( \ot \sigma  \id ) \oH
\id ) ( b \ast ( a \ast v ) ),
\end{equation*}
which is understood as in \eqref{jacobi}. An $L$-module $V$ is \emph{finite} if it is finitely generated as an $H$-module.

Let $V$ be a representation of the Lie pseudoalgebra $L$. If $A \subset L, X \subset V$, then set $A\cdot X$ to be the smallest $H$-submodule $N$ of $V$ such that $a \ast v \in (H \tt H) \oH N$ for all $a \in A, v \in X$. By Remark \ref{coeff}, $A \cdot X$ is the $H$-submodule of $V$ generated by all coefficients of $a \ast v, a \in A, v \in X$.

An $H$-submodule $W\subset V$ is {\em stable under the action of} $a \in L$ if $a \cdot W \subset W$. It is an \emph{$L$-submodule} of $V$ if $L \cdot W \subset W$. An $L$-module $V$ is \emph{irreducible} if it does not contain any nontrivial $L$-submodule. If $U$ and $V$ are two $L$-modules, then a map $\phi: U \longrightarrow V$
is a \emph{homomorphism} of $L$-modules if it is $H$-linear and satisfies
$$a \ast \phi (u)=  ( ( \ot \id  \id ) \oH \phi ) ( a \ast u ),$$
for all $a \in L, \, u \in U.$
Let $V,W$ be two $H$-modules. A map $f: V \longrightarrow (\ot H H) \oH W$ is a \emph{pseudolinear map} from $V$ to $W$ if it is $\kk$-linear and satisfies
$$ f(hv) = (1 \tt h) \cdot f(v), \quad \quad h \in H, v \in V.$$
The space $\chom (V,W)$ of all pseudolinear maps from $V$ to $W$ has a left $H$-module structure given by
$$ (h f) (v) = (h \tt 1) \cdot f(v).$$
If $V=W$ we set $\cend V = \chom (V,V)$. If $V$ is a finite $H$-module then there exists a unique Lie pseudoalgebra structure on $\cend V$ making $V$ a representation of $\cend V$ via the action $f \ast v = f (v)$. This Lie pseudoalgebra is usually denoted by $\gc V$, and making a finite $H$-module $V$ into a representation of a Lie pseudoalgebra $L$ is equivalent to giving a Lie pseudoalgebra homomorphism from $L$ to $\gc V$.

\begin{ex}
Any Lie pseudoalgebra $L$ over $H$ is a module over itself via $a * b := [a * b], a, b \in L$. When $L$ is finite, the adjoint action defines a Lie pseudoalgebra homomorphism $\ad : L \to \gc L$ whose kernel equals $Z(L)$. Notice that $L$ is nilpotent if and only if $L/Z(L)$ is nilpotent.
\end{ex}
\begin{rem}\label{torLtorM}
If $f \in \chom (V, W)$, then $f \ast v = 0$ as soon as $v \in \Tor V$. The adjoint action of any given $a \in L$ induces an element $\overline{\ad a} \in \chom(L^{[n]}/L^{[n+1]}, L^{[n+1]}/L^{[n+2]})$. Assume $L$ is a finite Lie pseudoalgebra, and $L^{[n]}, L^{[n+1]}$ have the same rank. Then the quotient $L^{[n]}/L^{[n+1]}$ is torsion, whence $\overline{\ad a} = 0$ for all $a \in L$. This forces $L^{[n+1]} = L^{[n+2]}$, and we conclude that the central series of any finite Lie pseudoalgebra stabilizes to an ideal, that we denote by $L^{[\infty]}$.
\end{rem}

\subsection{Action of coefficients}

If $a, b$ are elements of a Lie pseudoalgebra $L$ over $H$, it may be useful to know the action of coefficients $[a \ast b]_\gamma$, as defined in Corollary \ref{coefficienti}, on an $L$-module $M$.
\begin{lemma}\label{actioncoeff}
Let $L$ be a Lie pseudoalgebra over $H$, $M$ an $L$-module. Choose $a, b \in L$, and set $\alpha = [a \ast b]$. Assume that
$$[a \ast b] \ast v = \sum_i (k^i \tt l^i \tt m^i) \oH v_i,$$
where $k^i, l^i, m^i, \in H$ and $v_i \in M$. If $\gamma \in \Hom_\kk(H, H)$, then
\begin{equation}\label{actionofcoeff}
\alpha_\gamma \ast v = \sum_i (\gamma(k^i S(l^i_{(1)})) l^i_{(2)}
\tt m^i) \oH v_i.
\end{equation}
\end{lemma}
\begin{proof}
The assignment $k \tt l \tt m \mapsto (\gamma(k S(l_{(1)})) l_{(2)} \tt m) \oH u$ extends  to a well-defined linear map $\phi_\gamma: (H \tt H \tt H) \oH M \to (H \tt H) \tt_H M$. Moreover, if $[a \ast b] = \sum_i (h^i \tt 1) \tt_H c_i$, where the $h^i$ are linearly independent, and $c_i \ast v = \sum_j (k^{ij} \tt l^{ij}) \oH
v_{ij}$, then
\begin{equation}\label{actionofbracket}
[a \ast b] \ast v = \sum_{i, j} (h^i S(k^{ij}_{(1)}) \tt k^{ij}_{(2)} \tt l^{ij}) \oH v_{ij}.
\end{equation}
If we choose $\gamma_i(h^j) = \delta_i^j$, which is possible by linear independence of elements $h^j$, then $\phi_{\gamma_i}$ recovers from \eqref{actionofbracket} the expression \eqref{actionofcoeff} for the action of $c_i = c_{\gamma_i}$ on $v$. The general statement follows by $H\tt H$-linearity of the pseudobracket.
\end{proof}

\section{Representations of solvable and nilpotent Lie pseudoalgebras}
In the following two sections we recall some results from \cite{BDK1} about representation theory of solvable and nilpotent Lie pseudoalgebras.

\subsection{Weight vectors}
Let $L$ be a Lie pseudoalgebra over $H$ and $M$ be an $L$-module. If $\phi \in \Hom_{H} (L,H)$, the \emph{weight space} $M_{\phi}$ is defined as
$$M_{\phi} =\{ v \in M \mid a \ast v = ( \ot {\phi(a)}{1} ) \oH v,
\,\, \mbox{ for all } a \in L \}.$$
If $M_{\phi} \neq 0$, then $\phi$ is a {\em weight} for the action of $L$ on $M$. Every nonzero element of $M_{\phi}$ is a \emph{weight vector} of weight $\phi$.
\begin{rem}\label{freeness}
The weight space $M_{0}$ is always an $H$-submodule of $M$, whereas $M_\phi, \phi \neq 0$ is just a vector subspace. However, in this case, the $H$-submodule $H M_{\phi}\subset M$ is free over $M_{\phi}$.
\end{rem}
We have the following pseudoalgebraic analogues of Lie's Theorem:
\begin{thm}\label{liethm}
Let $L$ be a solvable Lie pseudoalgebra over $H$. Then every finite non-trivial $L$-module has
a weight vector.
\end{thm}
\begin{cor}\label{corLie}
If $L$ is a solvable Lie pseudoalgebra over $H$ and $M$ is a finite $L$-module then $M$ has a finite filtration by $L$-submodules $(0)=M_0 \subset M_1 \subset \cdots \subset M_n=M$ such that each quotient $M_{i+1} / M_i$ is generated over $H$ by a weight vector for the action of $L$.
\end{cor}
The {\em length} of an $L$-module $M$ is the minimal length of a filtration as above.

\subsection{Generalized weight submodules}
Let $L$ be a Lie pseudoalgebra over $H$, $\phi \in \Hom_H (L,H)=L^*$.
\begin{lemma}\label{MphiLsubmod}
Let $M$ be a finite $L$-module. If $N\subset M_{\phi}$ is a vector subspace, then $HN$ is an $L$-submodule of $M$.
\end{lemma}
\begin{proof}
Let $n \in N$. Then, for any $h \in H$, $a * hn = (\ot 1h) (a *n)= (\phi (a) \otimes h) \oH n \in (\ot HH) \oH HN.$
\end{proof}
We set $M^{\phi}_{-1}= (0)$ and inductively
\begin{equation*}
M^{\phi}_{i+1}=\spn_H \{ m \in M \mid a * m - (\phi(a) \otimes 1) \oH m \in (\ot HH) \oH M^{\phi}_{i}, \,\, \forall \, a \in L \}.
\end{equation*}
Then $M^{\phi}_{0}=H M_{\phi}$ and $M^{\phi}_{i+1} / M^{\phi}_{i}= H {( M/ M^{\phi}_{i})}_{\phi}$. The $M^{\phi}_{i}$ form an increasing sequence of $H$-submodules of $M$. By Noetherianity of $M$ this sequence stabilizes to an $H$-submodule $M^{\phi}=\bigcup\limits_i M^{\phi}_{i}$ of $M$, which is called the \emph{generalized weight submodule} relative to the \emph{weight} $\phi$.

We will occasionally stress the dependence of $M^\phi$ on the choice of the Lie pseudoalgebra acting on $M$ by writing $M_L^\phi$, or simply $M_a^\phi$ when $L = \langle a \rangle$ is the subalgebra of $\gc M$ generated by a single element $a$.
\begin{prop}
Let $N \subset M$ be finite $L$-modules, $\phi, \psi\in L^*$. Then
\begin{itemize}
\item $M^\phi$ is an $L$-submodule of $M$, and is a free $H$-module whenever $\phi \neq 0$.
\item The sum of generalized weight submodules of $M$ is always direct. In particular $M^\phi \cap M^\psi = (0)$ if $\phi \neq \psi$.
\item $N \subset M^\phi$ if and only if $N^\phi = N$.
\item $(M/M^\phi)^\phi = (0)$.
\item If $N \subset M^\phi$, then $N = M^\phi$ if and only if $(M/N)^\phi = (0)$.
\end{itemize}
\end{prop}
\begin{proof}
$M^\phi$ is an $L$-submodule by construction. It can be shown to be a free $H$-module by induction, using Remark \ref{freeness}.
All other statements follow easily from their Lie theoretic analogues by reinterpreting $M$ as a representation of the annihilation Lie algebra ${\mathcal L} = H^*\oH L$, see \cite{BDK1}.
\end{proof}

Notice that the direct sum $\sum\limits_{\phi \in L^*} M^\phi$ may fail to equal $M$. Equality, however, always holds when $L$ is nilpotent.

\begin{thm}\label{nilpotentdecomposition}
Let $L$ be a nilpotent Lie pseudoalgebra over $H$ and $M$ be a (faithful) finite $L$-module. Then $M$ decomposes as a direct sum of its generalized weight submodules, $M=\bigoplus\limits_{\phi \in L^*} M^{\phi}$.
\end{thm}

%%%%%%%%%%%%%%%%%%%%%%%%%%%%%

\subsection{Nilpotent pseudoalgebras}

We aim to show that the converse to Theorem \ref{nilpotentdecomposition} also holds, at least when $L$ is finite. We will first prove the result when $M$ coincides with one of its generalized weight spaces with respect to the action of $L$. The general statement will then follow easily.
\begin{prop}\label{Lnilpzero}
Let $M$ be a finite $H$-module and $L\subset \gc M$ be a Lie pseudoalgebra, $\phi \in L^*$, and assume that $M$ coincides with its $\phi$-generalized weight space with respect to the action of $L$. Then $L$ is a nilpotent Lie pseudoalgebra.
\end{prop}
\begin{proof}
Let $L^{[0]}=L$, $L^{[i]}=[L,L^{[i-1]}]$, $i \geq 1$, be the the central series of $L$ and $\{M_{k}\}$ be an increasing family of $L$-submodules of $M$ as in Corollary \ref{corLie}.

As a warm up, let us treat the case $\phi \equiv 0$ first. An easy induction shows that $L^{[i]} \cdot M_{k} \subset M_{k-i-1}$. Indeed, as $M = M^0$, we have $L \cdot M_k \subset M_{k-1}$ by our choice of $M_k$. Since $L^{[0]} = L$, this takes care of the basis of induction $i=0$.

Assume now that $L^{[i]} \cdot M_k\subset M_{k-i-1}$. Then
\begin{equation*}
\begin{array}{ll}
L^{[i+1]} \cdot M_k & = [L,L^{[i]}] \cdot M_k =L \cdot (L^{[i]} \cdot M_k) + L^{[i]} \cdot (L \cdot M_k)\\
& \subset L \cdot M_{k-i-1} + L^{[i]} \cdot M_{k-1} \subset M_{k-i-2}.
\end{array}
\end{equation*}
In order to conclude the proof it is enough to observe that if $M = M_n$ then, for $i=n$, we have $L^{[n-1]} \cdot M_n \subset M_0 = (0)$. This implies $L^{[n-1]}=0$, i.e., $L$ is a nilpotent Lie pseudoalgebra.\\

If $\phi \not\equiv 0$, the situation is slightly more delicate. For $1 \leq r \leq m$, let $m_r$ denote the cyclic generators of the quotients $M_r/M_{r-1}$. Set
\begin{equation*}
N^j_k=\{ b \in L \mid b * m_r \in (H \otimes F^{j}H) \oH m_{r-k} + M_{r-k-1}, \, r=1,\dots,N \},
\end{equation*}
and $N_k=\bigcup\limits_{j\in \N} N^j_k$, so that $N_{k}^j\subset N_{k+1}$ if $j<0$. Notice that $N_0 = L, N_{n+1} = (0)$.\\

We aim to prove, by induction on $p$, that $[L^{[p]}, N^j_k] \subset N^{j-p-1}_{k}$. Let us start with the basis of our induction: $[L^{[0]}, N^j_k] = [L, N^j_k] \subset N^{j-1}_k$. Let $a \in L$, $b \in N^j_k$. We know that
$$a*m_r=(\phi \otimes 1) \oH m_r \mod M_{r-1},$$
for $r=1,\dots,n$. Assume that $b*m_r=\sum\limits_i (h^i \otimes k^i) \oH m_{r-k} \mod M_{r-k-1}$, where $k^i \in F^{j}H$ for all $i$. Let us compute
\begin{equation*}
\begin{array}{rl}
[a*b]*m_r & = a*(b*m_r) -  (( \ot \sigma  \id ) \oH \id ){(b*(a*m_r))}\\
& = a* (\sum\limits_i (h^i \otimes k^i) \oH m_{r-k}) - (( \ot \sigma  \id ) \oH \id ) {(b*((\phi \otimes 1) \oH m_r ))}\\
%& = \sum\limits_i (1 \otimes h^i \otimes k^i)(1 \otimes \Delta) (a * m_{i-k}) -  (( \ot \sigma  %\id ) \oH \id ) \sum\limits_i {((1 \otimes \phi \otimes 1)(1 \otimes \Delta) (b* m_i) )}\\
& = \sum\limits_i(\phi \otimes h^i \otimes k^i - \phi k^i_{(1)} \otimes h^i \otimes k^i_{(2)} ) \oH m_{r-k}
\end{array}
\end{equation*}
up to terms in $a \cdot M_{r-k-1} + b \cdot M_{r-1} \subset M_{r-k-1}$. Now recall that, as $k^i \in F^j H$, then $\Delta(k^i) - 1 \tt k^i \in H \tt F^{j-1} H$. We conclude that all coefficients of $[a*b]$ lie in $N^{j-1}_{k}$.\\

As for the inductive step, assume now that $[L^{[p]}, N^j_k] \subset N^{j-p-1}_k$. Then we have
\begin{equation*}
\begin{array}{ll}
[L^{[p+1]}, N^j_k]  = [[L, L^{[p]}], N^j_k] & \subset [L, [L^{[p]}, N^{j}_k] + [L^{[p]}, [L, N^j_k]] \\
&\subset [L, N^{j-p-1}_{k}] + [L^{[p]}, N^{j-1}_k] \subset N^{j-p-2}_k.
\end{array}
\end{equation*}
Since $L$ is a finite Lie pseudoalgebra, there exists $d$ such that $N_k=N^d_k$ for all $k$. Then we obtain
$[L^{[d]}, N_k] \subset N^{-1}_{k} \subset N_{k+1}$. As a consequence, $L^{[(n+1)d + 1]} = [L^{[(n+1)d]}, N_0] =(0)$, which proves that $L$ is a nilpotent Lie pseudoalgebra.
\end{proof}
\begin{thm}\label{nilpweightsubmod}
Let $M$ be a finite faithful module over a finite Lie pseudoalgebra $L$, and assume that $M=\bigoplus\limits_{\phi \in L^*} M^{\phi}$. Then $L$ is a nilpotent Lie pseudoalgebra.
\end{thm}
\begin{proof}
Observe that under the assumption $M=\bigoplus\limits_{\phi \in L^*} M^{\phi}$ we have $L \subset \bigoplus\limits_{\phi \in L^*} \gc (M^{\phi})$. By Proposition \ref{Lnilpzero}, the image $L_{\phi}$ of $L$ in $\gc (M^{\phi})$ is nilpotent for all $\phi \in L^*$.  As a consequence, $\bigoplus\limits_{\phi \in L^*} L_{\phi}$ is nilpotent as it is a finite sum of nilpotent Lie pseudoalgebras. Finally, $L$ is a nilpotent Lie pseudoalgebra as it embeds in $\bigoplus\limits_{\phi \in L^*} L_{\phi}$.
\end{proof}
\begin{ex}
The finiteness assumption on $L$ in the statement of Theorem \ref{nilpweightsubmod} cannot be removed.

Indeed, let $M=Hm_1+Hm_2$ be a free $H$-module of rank $2$, and choose $L \subset \gc M$ to be the Lie pseudoalgebra of all pseudolinear maps $A \in \gc M$ such that $A*m_1 = (\phi(A) \tt 1) \oH m_1$, $A*m_2 = (\phi(A) \tt 1) \oH m_2 \mod (H \tt H)\oH m_1$, for some $\phi(A)\in H$. Then the central series of $L$ stabilizes to $L'$, which contains all $A$ such that $\phi(A) = 0$.
\end{ex}

Later on, we will deal with finite vertex algebras, and the following pseudoalgebraic analogue of Engel's theorem will turn out to be useful.
\begin{thm}\label{engel}
Let $L$ be a finite Lie pseudoalgebra over $H$. Assume that, for every $a\in L$, the generalized weight submodule $L_a^0$ for the adjoint action of $\langle a \rangle$ equals $L$. Then $L$ is a nilpotent Lie pseudoalgebra.
\end{thm}

\section{Approximate nilpotence of solvable subalgebras of $\gc M$}\label{solvanilpmod}
In this section we present the following result for $1$-generated solvable subalgebras of $\gc M$:
\begin{thm}\label{quasinilpo}
Let $M$ be a finitely generated $H$-module. If $a \in \gc M$ generates a solvable subalgebra $S = \langle a \rangle$, then there exists $\bar a \in S$, $\bar a \equiv a \mod S'$, such that the subalgebra $\langle \bar a \rangle$ is nilpotent.
\end{thm}
We will later specialize this result to give a characterization of finite vertex algebras.

\subsection{The length $2$ case}%\label{quasinilpres}
Let $M$ be a finite $H$-module and $a \in \gc M$ an element generating a solvable Lie pseudoalgebra $\langle a \rangle=S$.\\
A \emph{modification} of $a \in S$ is an element $\bar{a} \in S$ such that $a \equiv \bar{a} \,\, \mod S'$.
It follows by definition that the subalgebra generated by $\bar{a}$ is a subalgebra of $S$. The same inclusion holds for the corresponding derived subalgebras. As a consequence, a modification of a modification of $a$ is still a modification of $a$.
\begin{rem}
Let $\phi \in S^*$ be a weight for the action of $S = \langle a \rangle$ on $M$. Then $S = Ha + S'$, and the restriction of $\phi$ to $S'$ vanishes. This means that $\phi$ is uniquely determined by $\phi(a)$. As a consequence, $\phi(a)=\phi(\bar{a})$ whenever $\bar{a}$ is a modification of $a$.
\end{rem}
\begin{rem}
Let $M$ be a finite $H$-module, $S$ be a solvable Lie pseudoalgebra generated by $a \in \gc M$ and $N \subset M$ an $S$-submodule. Then $N$ is stable under the action of any modification $\bar{a}$ of $a$, as $\bar a$ belongs to $S$.
\end{rem}
\begin{prop}\label{phipsinonzero}
Let $M$ be a finite $H$-module and $a \in \gc M$ such that $S=\langle a \rangle$ is a solvable subalgebra. Assume that $M = Hu + Hv$, where $u \in M$ is a $\phi-$weight vector and $[v]\in M/Hu$ is a $\psi-$weight vector for the action of $S$, for some $\phi \neq \psi \in S^*$. Then there exists a lifting $\bar v \in M$ of $[v]$ such that $H \bar v$ is a complement of $Hu$ in $M$ and is stable under the action of some modification $\bar a$ of $a$.
\end{prop}
\begin{lemma}\label{tech}
Under the same hypotheses as above, let $b \in \gc M$ be such that $b\ast u = 0, b \ast v = (\beta \tt k) \tt_H u$, where $\beta, k \in H$, and the degree of $k$ is $K$. Then some coefficient $s$ of $[a \ast b]$ satisfies $s \ast u = 0$, $s \ast v = (1 \tt k) \tt_H u \mod (H \tt F^{K-1}H) \oH u$.
\end{lemma}
\begin{proof}
A direct computation gives
\begin{equation*}
\begin{array}{rl}
[a*b] * v & = (\phi(a) \otimes \beta \otimes k- \psi(a) k_{(1)} \otimes \beta \otimes k_{(2)})\oH u\\
& = \qquad (\alpha \tt \beta \tt k) \oH u,
\end{array}
\end{equation*}
up to terms in $(H \tt F^{K-1}H) \oH u$. By Lemma \ref{actioncoeff}, we have
\begin{equation*}
[a \ast b]_\gamma \ast v = \gamma(\alpha S(\beta_{(1)}))\beta_{(2)} \tt k) \oH u \mod (H \tt F^{K-1}H) \oH u.
\end{equation*}
It now suffices to choose $\gamma \in \Hom_\kk(H, H)$ so that $\gamma(\alpha S(\beta))$ equals $1$, and requiring that it vanish on all terms of lower degree.
\end{proof}

\begin{proof}[Proof of Proposition \ref{phipsinonzero}]
We know that
\begin{equation*}
\begin{array}{lll}
a* u= & (\phi(a) \otimes 1) \oH u & \quad\\
a* v= & (\psi(a) \otimes 1) \oH v & \mod (H \tt H) \oH u.
\end{array}
\end{equation*}
We may then find $K \in \N$, and linearly independent elements $k^i \in H$ of degree $K$, such that
\begin{equation*}
a* v= \sum_i (h^i \otimes k^i) \oH u + (\psi(a) \otimes 1) \oH v \mod (H \tt F^{K-1} H) \tt_H u.
\end{equation*}
We may assume that the $h^i$ are linearly independent as well.

By a direct computation we obtain, modulo terms in $(H \tt H \tt F^{K-1} H) \oH u$,
\begin{equation*}
\begin{array}{rl}
a*(a*u) = & (\phi(a) \tt \phi(a) \tt 1) \oH u,\\
a*(a*v)= & \sum\limits_i ( \phi(a) \otimes h^i \otimes k^i ) \oH u + \sum\limits_i (h^i \otimes \psi(a) k^i_{(1)} \otimes  k^i_{(2)}) \oH u + (\psi(a) \otimes \psi(a) \otimes 1) \oH v\\
= & \sum\limits_i (\phi(a) \otimes h^i \otimes k^i ) \oH u + \sum\limits_i (h^i \otimes \psi(a) \otimes  k^i) \oH u + (\psi(a) \otimes \psi(a) \otimes 1) \oH v,
\end{array}
\end{equation*}
so that $[a \ast a] * u = 0$ and
\begin{equation*}
[a *a] * v = \sum_i (\alpha \otimes h^i \otimes k^i - h^i \otimes \alpha \otimes k^i ) \oH u \mod (H \tt H \tt F^{K-1}H)\oH u,
\end{equation*}
where $\alpha = \phi(a) - \psi(a) \in H$ is a nonzero element of degree $N$. Let $D$ be the maximal degree of the $h^i$.

We now proceed by induction on $K$, on $D$ and on the rank of $\sum_i h_i \tt k_i \in H \tt H$. We distinguish three cases:
\begin{enumerate}
\item $D > N$. Then we may choose $\gamma\in \Hom_\kk(H, H)$ such that $\gamma(\alpha) = 1$ and obtain
\begin{equation*}
\begin{array}{rl}
[a\ast a]_\gamma \ast v = & \sum_i (\gamma(\alpha S(h^i_{(1)}))h^i_{(2)} \tt k^i - \gamma(h^iS(\alpha_{(1)})
\alpha_{(2)} \tt k^i) \tt_H u,\\
= & \sum_i (h^i \tt k^i) \oH u
\end{array}
\end{equation*}
modulo terms in $(F^{N-1} H \tt F^K H + H \tt F^{K-1}H)\tt_H u$. The modification $a - [a \ast a]_\gamma$ then leads to a coefficient $\sum_i h^i \tt k^i$ of lower degree in the first tensor factor.\\

\item $N \geq D$ and $\alpha \notin \spn_\kk\langle h^i\rangle$. Choose $\gamma_i$ such that $\gamma_i(h^i) = -\delta_i^j, \gamma(\alpha) = 0$. Then
\begin{equation*}
\begin{array}{rl}
[a\ast a]_{\gamma_i} \ast v = & \sum_i (\gamma(\alpha S(h^i_{(1)}))h^i_{(2)} \tt k^i - \gamma(h^iS(\alpha_{(1)})) \alpha_{(2)} \tt k^i) \tt_H u,\\
= & (\alpha \tt k^i) \tt_H u,
\end{array}
\end{equation*}
modulo terms in $(F^N H \tt k^i + H \tt F^{K-1}) \oH u$. This shows that some coefficient $b$ of $[a \ast a]$ acts on $v$ so that
$$b \ast v = (\beta \tt k^i) \tt_H u \mod (H \tt F^{K-1}H)\oH u,$$
for some nonzero $\beta \in H$. By Lemma \ref{tech} we may then find, for each $i$, some element $s_i \in S'$ such that
$$s_i \ast v = (1 \tt k^i) \oH u \mod (H \tt F^{K-1}H) \oH u.$$
The element $a - \sum_i h^i s_i$ is then a modification of $a$ leading to a lower value of $K$.\\

\item $N \geq D$ and $\alpha \in \spn_\kk\langle h^i \rangle$. In this case we can find $c_i \in \kk$ such that $\alpha = \sum_i c_i h^i$. Choose $j$ so that $c_j \neq 0$, and set $v' = v - c_j^{-1} k^j u$. Then
\begin{equation*}
\begin{array}{ll}
a*v' & =-(\phi(a) \otimes c_j^{-1} k^j) \oH u + \sum\limits_i (h^i \otimes k^i)\oH u + (\psi(a) \otimes 1) \oH v\\
& = (-\phi(a) \otimes c_j^{-1} k^j + \sum\limits_i h^i \otimes k^i)\oH u + (\psi(a) \otimes 1) \oH (c_j^{-1} k^j u)+ (\psi(a) \otimes 1) \oH v'\\
& = (-c_j^{-1}\phi(a) \otimes  k^j + \sum\limits_i h^i \otimes k^i)\oH u + (c_j^{-1}\psi(a)k^j_{(1)} \otimes k^j_{(2)}) \oH u+ (\psi(a) \otimes 1) \oH v'\\
& =((h^j - c_j^{-1}\alpha ) \otimes k^j + \sum\limits_{i \neq j} (h^i \otimes k^i) ) \oH u + (\psi(a) \otimes 1) \oH v',
\end{array}
\end{equation*}
modulo terms in $(H \tt F^{K-1} H) \oH u$. The element $h^j - c_j^{-1}\alpha$ is a linear combination of the $h^i, i \neq j$, and so the rank of the $H \tt H$-coefficient multiplying $u$ is lower than that of $\sum_i h^i \tt k^i$.
\end{enumerate}
We may now apply induction.
\end{proof}
\begin{rem}
Notice that, in the above proof, the coefficient $\sum_i h^i \tt k^i$ is not uniquely determined, in case $u$ is a torsion element of $M$. However, the proof works equally well for any given choice of such a coefficient.
\end{rem}

\subsection{Proof of the general statement}

\begin{prop}\label{phipsi}
Let $M$ be a finite $H$-module and $a \in \gc M$ such that $S=\langle a \rangle$ is a solvable subalgebra. Assume that $\phi \neq \psi \in S^*$ are such that $M/M^{\phi}={(M/M^{\phi})}^{\psi}$. Then there exists an $H$-submodule $\overline M \subset M$ which is a complement to $M^\phi$ and is stable under some modification of $a$.
\end{prop}
\begin{proof}
We start by considering the case when the length of the $S$-module $M/M^\phi$ is $1$, and proceed by induction on the length $n$ of $M^\phi$. The basis of induction $n=1$ is provided by Proposition \ref{phipsinonzero}, so we assume that the length of $M^\phi$ equals $n> 1$.

Choose $u \in M_\phi$ such that $M^\phi/Hu$ has length $n-1$. We use induction on the $S$-module $M/Hu$ to find a complement $N/Hu$ to $M^\phi/Hu = (M/Hu)^\phi$ which is stable under the action of some modification $\tilde a$ of $a$. Notice that $N^\phi = N \cap M^\phi = Hu$ and that $N/Hu$ is isomorphic to $(M/Hu)/(M^\phi/Hu) \simeq M/M^\phi$, hence we may apply Proposition \ref{phipsinonzero} to $N$ and find a complement $M' \subset N$ to $Hu$ which is stable under some modification $\bar a$ of $\tilde a$. Now, $M^\phi + M' = M^\phi + Hu + M' = M^\phi + N = M$; moreover $M^\phi \cap M' \subset M^\phi \cap N = Hu$ so that $M^\phi \cap M' \subset M' \cap Hu = (0)$. We
conclude that $M'$ is a complement of $M^\phi$ in $M$ that is stable under the action of $\bar a$, which is a modification of $a$.\\

We proceed now with proving the statement when the length $m$ of $M/M^\phi$ is greater than $1$. Choose $\overline N = N/M^\phi \subset M/M^\phi$ of length $m-1$ so that $M/N$ has length $1$; as $M/M^\phi = (M/M^\phi)^\psi$, then $\overline N = \overline N^\psi$. Since $N^\phi = N \cap M^\phi = M^\phi$, we may use induction to find an $H$-submodule $N' \subset N$ which is a complement to $M^\phi$ and is stable under some modification $a'$ of $a$.

Consider now the quotient $M' = M/N'$. Then $(M')^\phi$ certainly contains the image $(M^\phi + N')/N'$ of $M^\phi$ under the canonical projection $\pi: M \to M/N'$. Moreover, $(M/N')/((M^\phi + N')/N')$ is isomorphic to $M/(M^\phi + N')$ and is therefore a quotient of $M/M^\phi$, which equals its $\psi$-generalized weight space. As $\psi \neq \phi$, we conclude that $(M')^\phi = (M^\phi + N')/N'$, and that $M'/(M')^\phi \simeq M/(M^\phi + N') = M/N$ has length one. We may then find a complement $\overline{M}/N'$ of $(M')^\phi$ in $M'$ which is stable under some
modification $\bar a$ of $a'$. We claim that $\overline M$ is a complement of $M^\phi$ in $M$.

Indeed, $\overline{M}/N' + (M')^\phi = M'$, hence $\overline{M} + (M^\phi + N') = M$; as $N' \subset \overline{M}$, we conclude that $M = \overline{M} + M^\phi$. On the other hand, $\overline{M} \cap M^\phi = N'$, hence $\overline{M} \cap M^\phi \subset N' \cap M^\phi = (0)$.
\end{proof}
We are now ready to prove our central result.
\begin{prop}\label{quasinilp}
Let $M$ be a finite $H$-module and $S$ be a solvable Lie pseudoalgebra generated by $a \in \gc M$. Then there exists a modification $\bar{a}$ of $a$ such that $M$ decomposes as a direct sum of generalized weight modules with respect to $\overline S = \langle \bar a\rangle$.
\end{prop}
\begin{proof}
By induction on the length of $M$. If the length equals $1$, then $M = M^\phi$ for some $\phi$ and there is nothing to prove.

Let us assume that the length of $M$ is $n>1$. We may find a weight vector $u \in M_\phi$ such that $N = M/Hu$ has length $n-1$. By inductive assumption, $N$ decomposes as a direct sum $N^{\phi_1} \oplus \dots \oplus N^{\phi_r}$ of (non trivial) generalized weight modules with respect to the subalgebra $\widetilde S\subset S$ generated by some modification $\tilde a$ of $a$. Let $N^i$ be the preimage of $N^{\phi_i}$ under the canonical projection $\pi: M \to M/Hu$, and reorder indices so that $\phi_i \neq \phi$ for all $i \neq r$.

As long as $\phi_k\neq \phi$, we may repeatedly apply Proposition \ref{phipsi} to obtain complements $M^i$ to $(N^i)^\phi = Hu$ in $N^i$ so that the sum $M^1 + \dots + M^k$ is direct and all summands are invariant with respect to some iterated modification of $a$. If $\phi_r \neq \phi$ holds as well, we end up with $M = M^1 \oplus \dots \oplus M^{r-1} \oplus M^r \oplus Hu$; if instead $\phi_r = \phi$, then $M = M^1 \oplus \dots \oplus M^{r-1} \oplus N^r$. In both cases, all summands are generalized weight spaces by construction, and are stable with respect to some modification $\bar a$ of $a$.
\end{proof}
In the light of Theorem \ref{nilpweightsubmod}, we see that Theorem \ref{quasinilpo} is just a restatement of Proposition \ref{quasinilp}.\\

Let $\bar a$ be a modification of $a$ generating a nilpotent subalgebra of $\gc M$. A natural question to ask is whether the decomposition $M = \bigoplus M_{\bar a}^\phi$ depends on $\bar a$ or is instead canonical. This amounts to asking if all such modifications of $a$ are contained in a single nilpotent subalgebra of $\langle a \rangle$. We will answer this in the negative at the very end of the paper.

\begin{cor}
Let $L$ be a Lie pseudoalgebra over $H$. If $a \in L$ generates a finite solvable subalgebra, then some modification of $a$ generates a nilpotent subalgebra.
\end{cor}
\begin{proof}
Let $S = \langle a \rangle$. The adjoint action of $S$ gives rise to a homomorphism $\ad: S \to \gc S$ of pseudoalgebras whose kernel equals the centre $Z(S)$ of $S$. Moreover, $\ad S$ is a solvable subalgebra of $\gc S$ generated by $\ad a$. By Theorem \ref{quasinilpo} we may find in $\ad S$ a modification of $\ad a$ generating a nilpotent subalgebra $N$ of $\gc S$. Such a modification is of the form $\ad \bar a,$ where $\bar a$ is a modification of $a$. Then $N$ is isomorphic to the quotient of $\overline S = \langle \bar a \rangle$ by a central --- as it is contained in $Z(S)$ --- ideal. We conclude that $\overline S$ is nilpotent.
\end{proof}

Theorem \ref{quasinilpo} has some interesting consequences.
\begin{prop}
Let $M$ be a finite $H$-module, $S$ be a solvable Lie pseudoalgebra acting on $M$. If $\phi \in S^*$ is nonzero and $U \subset M^\phi$ is an $S$-submodule, then there exists an $H$-linear section $s: M/U\to M$. In particular:
\begin{itemize}
\item If $M$ is torsion-free, then $M/U$ is torsion-free;
\item If $M/U$ is free, then $M$ is free.
\end{itemize}
\end{prop}
\begin{proof}
Choose $a\in S$ such that $\phi(a) \neq 0$ and $\langle a \rangle$ is nilpotent. We may then replace $S$ with $\langle a \rangle$, and assume $M = M^\phi \oplus M^{\phi_1} \oplus \dots \oplus M^{\phi_r}$. Then $M^\phi/U$ is free as an $H$-module, and we can find a section $s: M^\phi/U \to M^\phi$. This extends to a section $s: M/U \to M$ thanks to the direct sum decomposition.
\end{proof}

\begin{cor}
Let $M$ be a finite $H$-module, $S$ be a solvable Lie pseudoalgebra acting on $M$. If $M$ is not free, then some quotient of $M$ has a $0$-weight vector for the action of $S$.
\end{cor}

\section{Structure of finite vertex algebras}

\subsection{Preliminaries on vertex algebras}

Let $V$ be a complex vector space. A {\em (quantum) field} on $V$ is a formal power series $\phi(z) \in (\End V)[[z, z^{-1}]]$ such that $\phi(z) v \in V((z)) = V[[z]][z^{-1}]$. In other words,
$$\phi(z) = \sum_{n \in \Z} \phi_{(n)} z^{-n-1}$$
is a quantum field if and only if, for every choice of $v \in V$, $\phi_{(n)} v = 0$ for (depending on $v$) sufficiently high values of $n$.

A {\em vertex algebra} is a complex vector space $V$, endowed with a {\em vacuum vector} $\1 \in V$, an {\em infinitesimal translation operator} $T \in \End V$, and a $\C$-linear {\em state-field correspondence} $Y: V \to (\End V)[[z, z^{-1}]]$ mapping each element $a \in V$ to some field $Y(a,z)$ on $V$, satisfying, for all choices of $a, b \in V$,
\begin{itemize}
\item
$Y(\1, z)a = a, \qquad Y(a,z)\1 = a \mod zV[[z]]$; \hfill {\em (vacuum axiom)}
\item
$Y(Ta, z) = [T, Y(a,z)] = dY(a,z)/dz$; \hfill {\em (translation invariance)}
\item
$(z-w)^N [Y(a,z), Y(b,w)] = 0$, for some $N = N(a,b)$. \hfill{\em (locality)}
\end{itemize}
It is well known that commutators $[Y(a,z), Y(b,w)]$ may be expanded into a linear combination of the Dirac delta distribution
$$\delta(z-w) = \sum_{n \in \Z} w^n z^{-n-1},$$
and of its derivatives. More precisely, if
$$Y(a,z) = \sum_{n \in \Z} a_{(n)} z^{-n-1},$$
then
$$[Y(a,z), Y(b,w)] = \sum_{j=0}^{N(a,b)-1} \frac{Y(a_{(j)} b,
w)}{j!} \frac{d^j}{dw^j}\delta(z-w).$$
It is also possible to define the Wick, or {\em normally ordered}, product of quantum fields
$$: Y(a,z) Y(b,z) = Y(a,z)_+ Y(b,z) + Y(b,z) Y(a,z)_-,$$
where
$$Y(a,z)_- = \sum_{n \in \N} a_{(n)} z^{-n-1}, \qquad Y(a,z)_+ =
Y(a,z) -Y(a,z)_-.$$
Then one has
$$Y(a_{(-n-1)} b, z) = \frac{1}{n!}:Y(T^n a, z) Y(b,z):,$$
for all $n \geq 0$. One of the consequences of the vertex algebra axioms is the following:
\begin{itemize}
\item $Y(a,z)b = e^{zT} Y(b,-z)a$ \hfill{\em (skew-commutativity)}
\end{itemize}
for all choices of $a,b$.

Every vertex algebra has a natural $\C[T]$-module structure. A vertex algebra $V$ is {\em finite} if $V$ is a finitely generated $\C[T]$-module. A $\C[T]$-submodule $U \subset V$ is a subalgebra of the vertex algebra $V$ if $\1 \in U$ and $a_{(n)} b \in U$ for all $a, b \in U, n \in \Z$. Similarly, a $\C[T]$-submodule $I \subset V$ is an {\em ideal} if $a_{(n)} i \in I$ for all $a \in V, i \in I, n \in \Z$.

A subalgebra $U \subset V$ is {\em abelian} if $Y(a,z)b = 0$, or equivalently $a_{(n)} b = 0$, for all $a, b \in U, n \in \Z$. It is {\em commutative} if $[Y(a,z), Y(b,w)] = 0$, or equivalently $a_{(n)} b = 0$, for all $a, b \in U, n \in \N$.

Let $U$ be a subalgebra, and $I$ an ideal of a vertex algebra $V$; we say that $V$ is the {\em semidirect sum} of $U$ and $I$ (denoted $V = U \sd I$) if $V = U \oplus I$ is a direct sum of $\C[T]$-submodules. Every vertex algebra becomes a Lie conformal algebra, see \cite{DK}, after setting $\partial = T$ and
$$[a_\lambda b] = \sum_{n \in \N} \frac{\lambda^n}{n!} a_{(n)} b.$$
We have seen in Example \ref{conformal} that the notion of Lie conformal algebra is equivalent to that of Lie pseudoalgebra over $H = \CD$. In this setting, the pseudobracket is given by
$$[a \ast b] = \sum_{n \in \N} \left(\frac{(-\partial)^n}{n!} \tt 1\right)\tt_H a_{(n)} b.$$
If $V$ is a vertex algebra, we will denote by $V^\lie$ the underlying Lie conformal algebra structure.

\subsection{The nilradical}
In this section, we recall some properties of nilpotent elements in a vertex algebra. Proofs can be found in \cite{simple, varna}.

An element $a$ in a vertex algebra $V$ is {\em nilpotent} if $Y(a, z_1)Y(a, z_2) \dots Y(a, z_n)a = 0$ for sufficiently large values of $n$. An ideal $I \subset V$ is a {\em nil-ideal} if all of its elements are nilpotent; clearly, every abelian ideal of $V$ is a nil-ideal.
\begin{prop}
Let $V$ be a vertex algebra. Then
\begin{itemize}
\item Every nilpotent element of $V$ generates a nil-ideal.
\item The set $\Nil V$ of all nilpotent elements of $V$ is an ideal of $V$.
\item The vertex algebra $V/\Nil V$ contains no nonzero nilpotent elements.
\item If $V$ is finite, then $\Nil V$ is a nil-ideal of $V$.
\end{itemize}
\end{prop}
If $V$ is a finite vertex algebra, then $V^\lie$ is always a solvable Lie conformal algebra \cite{nilpo}. Recall that the central series of $V^\lie$ stabilizes, by Remark \ref{torLtorM}, to a vertex ideal $V^{[\infty]}$ of $V$. The following facts were proved in \cite{nilpo}.
\begin{prop}
Let $V$ be a finite vertex algebra, $S \subset V^\lie$ a subalgebra. Then $V_S^0$ is a subalgebra and $V_S^{\neq 0} = \sum\limits_{\phi\in S^*\setminus\{0\}} V_S^\phi$ is an abelian ideal of $V$.
\end{prop}
As a consequence, $V_S^{\neq 0} \subset \Nil V$. If $V$ is finite and $\Nil V = (0)$, then $V^\lie$ is nilpotent; if moreover $V$ is simple, then it is necessarily commutative.

\subsection{Root space decomposition of finite vertex algebras}

Let $V$ be a finite vertex algebra, $a \in V$. The subalgebra $S = \langle a \rangle\subset V^\lie$ is always solvable. The adjoint action of $S$ makes $V$ into a finite $S$-module, and we can find submodules $(0) = V_0 \subset V_1 \subset \dots \subset V_n = V$ as in Corollary \ref{corLie}.
The \emph{singularity}\footnote{The singularity of an element is, in other words, the multiplicity of the zero eigenvalue of its adjoint action.} of $a$ is then the number of non-torsion quotients $V_i/V_{i-1}$ with a trivial action of $S$. Notice that the singularity does not change under modifications of $a$. When $S$ is nilpotent, then the singularity of $a$ equals the rank of $V_a^0$ as an $H$-module.
\begin{thm}\label{fVAdec}
Let $V$ be a finite vertex algebra and $N=V^{[\infty]}$. Then $N$ is an abelian ideal of $V$, and there exists a subalgebra $U\subset V$ such that $U^{\lie}$ is nilpotent and $V=U \ltimes N$.
\end{thm}
\begin{proof}
Choose an element $a \in V$ of minimal singularity $k$. Up to replacing $a$ by a suitable modification, we may assume that $S = \langle a \rangle$ be a nilpotent subalgebra of $V^\lie$. Then, $V$ decomposes as a direct sum of generalized weight submodules,
\begin{equation*}
V=\bigoplus_{\phi\in S^*} V^{\phi}_{\bar{a}}= V_a^{0} \oplus V_a^{\neq 0}.
\end{equation*}
Then $U=V_a^{0}$ is a vertex subalgebra of $V$ and $N=V_a^{\neq 0}$ is an abelian ideal of $V$. We want to show that $U^\lie$ is nilpotent and $N = V^{[\infty]}$.

Say $b \in U$. As $U \subset V$ is a subalgebra, then $U$ is stable under the action of $b$. If the $U_b^0 \neq U$, then some (generic) linear combination of $a$ and $b$ would have lower singularity than $a$, a contradiction. Thus, all elements in $U^\lie$ have a nilpotent adjoint action, hence $U^\lie$ is nilpotent by Theorem \ref{engel}.

It remains to prove that $N=V^{[\infty]}$: since $N$ is an ideal such that $V^\lie/N$ is nilpotent then $V^{[\infty]}\subset N$. We prove the other inclusion by showing that $N \subset V^{[k]}$ by induction on $k\in \mb N$, the basis of the induction being clear, as $N \subset V^0=V $. By construction, $[a, V_a^\phi] = V_a^\phi$ if $\phi \neq 0$, hence $[a, N] = N$. Then $N \subset V^k$ implies $N = [a, N] \subset [V, V^{[k]}] = V^{[k+1]}$. We conclude that $N \subset V^{[n]}$ for all $n$, hence that $N \subset V^{[\infty]}$.
\end{proof}

\begin{rem}
In the above statement, $U$ is a vertex subalgebra of $V$ with the property that $U^\lie$ is a nilpotent and self-normalizing subalgebra of $V^\lie$. As a consequence, the adjoint action of $U^\lie$ on $V$ gives a generalized weight submodule decomposition in which the $0$-weight component is $U$ itself. It makes sense to call every such $U$ a {\em Cartan subalgebra} of $V$, and the corresponding decomposition a {\em root space decomposition}.

Notice that $N$ is the smallest nil-ideal of $V$ having a complementary subalgebra $U$ such that $U^{\lie}$ is nilpotent; as $N = V^{[\infty]}$, it is canonically determined. If $U$ is a Cartan subalgebra of $V$, then $U = V/V^{[\infty]}$, so all Cartan subalgebras of $V$ possess isomorphic vertex algebra structures.

We can be more precise. The identification of any two Cartan subalgebras $U, U'$ with $V/N$ gives an isomorphism $\phi: U \to U'$ which projects to the identity on $V/N$. If we extend $\phi$ to all of $V$ by setting it to be the identity on $N$, then we obtain an automorphism of $V$ conjugating $U$ to $U'$. Thus, all Cartan subalgebras of $V$ are conjugated under $\Aut V$.
\end{rem}

\subsection{A counterexample to nilpotence of finite vertex algebras}
The statement of Theorem \ref{fVAdec} suggests how to construct a finite vertex algebra $V$ such that the corresponding Lie conformal algebra $V^{\lie}$ is not nilpotent. What we need is a vertex algebra $U$ with a nilpotent underlying Lie conformal algebra $U^\lie$, and a suitable action on an $\Cd$-module $N$. The simplest case is when $U$ is a commutative vertex algebra, i.e., $U^{\lie}$ is abelian, and $N$ is a free $\Cd$-module of rank $1$.\\

Let $U=\{ a(t) \in \mb C [[t]] [t^{-1}] \}$. $U$ is a differential commutative associative algebra with $1$, with derivation $\partial=d/dt$. Hence $U$ has a commutative vertex algebra structure given by
\begin{equation}\label{FF}
Y(a(t),z) b(t)=(e^{z\partial} a(t))b(t)=i_{|z|<|t|} a(t+z)b(t),
\end{equation}
where $a(t),b(t) \in U$ and $i_{|z|<|t|}$ (see \cite{K}) indicates that one should expand $a(t+z)$ in the domain $|z|<|t|$, i.e., using positive powers of $z/t$.

Let $N=\Cd n$ be a free $\Cd$-module of rank $1$. We set:
\begin{equation}\label{NN}
Y(n,z)n=0,
\end{equation}
and define an action of $U$ on $N$ by setting
\begin{equation}\label{FN}
Y(a(t),z)n=a(z)n,
\end{equation}
where $a(t) \in U$.

\begin{thm}\label{mcFNVA}
There exists a unique vertex algebra structure on the $\Cd$-module $V=U \oplus N$ such that (\ref{FF}), (\ref{NN}), (\ref{FN}) are satisfied. Moreover, the central series of $V^{\lie}$ stabilizes to $N$.
\end{thm}
\begin{proof}
The unit element $\1 \in U$ satisfies the vacuum axiom, and all $Y(v, z), v \in V$ are fields by definition. Locality and translation invariance require some more effort. The skew-commutativity axiom suggests that we set
\begin{equation*}%\label{skecommcountex}
Y(n,z)a(t)=e^{z\partial}Y(a(t),-z)n=a(-z)e^{z\partial}n.
\end{equation*}
If further
\begin{equation*}%\label{trlinv}
Y(a(t),z) \partial^K n= \sum\limits_{i=1}^{K}{ K \choose i} {(-1)}^i a^{(i)}(z) \partial^{K-i}n,
\end{equation*}
then translation invariance is easily checked.

Let us move on to proving locality. First of all, notice that $Y(n, z)Y(n, w)$ maps every element of $V$ to $0$, hence $[Y(n, z), Y(n, w)] = 0$. Taking derivative with respect to $z$ and $w$, and using linearity, we obtain $[Y(u, z), Y(u', w)] = 0$ for all $u, u' \in N$.

Next, let us consider $[Y(a(t), z), Y(b(t), w)].$ An easy computation gives
\begin{equation*}
Y(a(t),z)(Y(b(t),w)\,\,n = Y(a(t),z)b(w)n = b(w)Y(a(t),z)n=a(z)b(w)n,
\end{equation*}
hence $[Y(a(t), z), Y(b(t), w)]n = 0,$ for all $a(t), b(t) \in U$. As
\begin{equation*}
\begin{array}{rl}
\partial ([Y(a(t), z), Y(b(t), w)]\,u) = & [Y(a'(t), z), Y(b(t), w)]\,u\\
+ & [Y(a(t), z), Y(b'(t), w)]\,u\\
+ & [Y(a(t), z), Y(b(t), w)]\,\partial u,
\end{array}
\end{equation*}
we conclude that $[Y(a(t), z), Y(b(t), w)]$ vanishes on all elements from $N$. However, it also vanishes on $U$, because of its vertex algebra structure.\\

We are left with showing that $[Y(a(t), z), Y(n, w)]$ is killed by a sufficiently large power of $z-w$. Let us compute
\begin{equation*}
\begin{array}{ll}
Y(a(t),z) (Y(n,w)b(t))&=Y(a(t),z)b(-w)e^{w\partial}n=b(-w) Y(a(t),z)e^{w\partial}n\\
& = b(-w)e^{w\partial} (e^{-w\partial}Y(a(t),z)e^{w\partial})n\\
&=e^{w\partial}b(-w) Y(a(t),z-w)n\\
&=i_{|w|<|z|} a(z-w) b(-w)e^{w\partial}n,\\
%\mbox{ and \hfill}&\\
Y(n,w) (Y(a(t),z) b(t))&=Y(n,w) i_{|z|<|t|} a(t+z)b(t)\\
&=i_{|z|<|t|} Y(n,w) a(t+z)b(t)\\
& = i_{|z|<|w|} a(z-w)b(-w)e^{w\partial}n.
\end{array}
\end{equation*}
Therefore,
\begin{equation*}
(z-w)^N [Y(a(t),z),Y(n,w)]b(t)= (i_{|w|<|z|} -i_{|z|<|w|})((z-w)^N a(z-w)b(-w)e^{w\partial}n)
\end{equation*}
is zero as soon as $t^N a(t)$ has no negative powers of $t$. As $[Y(a(t), z), Y(n, w)]$ maps every element of $N$ to zero, locality is then proved.\\

As for $V^{[\infty]}=N$, let $a=a(t)=\sum\limits_{n \in \mb Z} a_n t^{-n-1} \in U$ such that $a(t)$ contains some negative power of $t$. Then there exists $n \geq 0$ such that $a_n \neq 0$, hence $a(t) _{(n)} n=a_n \cdot n\neq0$. Therefore $[a, N] = N$, hence $N \subset V^{[n]}$ for all $n$. However, $(V/N)^\lie$ is nilpotent, hence $V^{[\infty]} \subset N$.
\end{proof}
Let us choose a finite subalgebra of $U$ whose conformal adjoint action on $N$ has nonzero weights, i.e., containing some element $a(t) \not\in \mb C[[t]]$.

\begin{ex}\label{counterex}
$M=\mb C[t^{-1}]\ltimes N \subset U \ltimes N$ is a finite vertex algebra, as it is generated over $\CD$ by $t^{-1}$, $1$ and $n$. However, $M^{\lie}$ is not nilpotent.
\end{ex}
We conclude by observing that even though the nil-ideal $N$ in the decomposition stated in Theorem \ref{fVAdec} is canonically determined, the subalgebra $U$ need not be. Indeed there may be several possible choices of $U$ as the following construction shows.

Let $M$ be as in Example \ref{counterex}, and choose $u \in N$. We know that $u_{(0)}$ is a derivation of $M$, and as $N$ is an abelian ideal of $M$, we immediately obtain $u_{(0)}^2=0$. Recall that the exponential of such a nilpotent derivation of a vertex algebra $M$ gives an {\em inner} automorphism of $M$.

If we choose $u = kn, k \in \kk$, then $\exp(kn_{(0)})(t^{-1})=t^{-1} -kn$. Thus, if we set $\psi=\exp(kn_{(0)})$, we obtain $\psi (N)=N$, $\psi(U)=\Cd(t^{-1}-kn)\oplus \mb C 1$, and $\psi (U)\cap N=\psi (U\cap N)=0$. We conclude that $\psi(U)$ is another subalgebra of $M$ which complements $N$.
Notice that in this example all Cartan subalgebras can be showed to be conjugated by an inner automorphism of $V$. It is not clear whether this holds in general.

One final comment is in order: it is easy to show that the Lie conformal subalgebra of $M^\lie$ generated by the element $a = t^{-1} + \partial n$ is solvable and equals $\Cd a + \Cd n$. As $[a_\lambda a] = (\partial + 2 \lambda) n$, we see that all elements $t^{-1} - kn, k \in \kk$ are modifications of $a$, and they generate nilpotent subalgebras of $M^\lie$. However, they $\Cd$-linearly span all of $\langle a \rangle$, whose Lie conformal algebra structure is solvable but not nilpotent. As the adjoint homomorphism $\ad: M^\lie \to \gc M$ is injective on $\langle a \rangle$, we conclude that there is no single nilpotent subalgebra of $\gc M$ containing all of the above modifications of $a$.

\vfill
\end{document}